\documentclass[12pt,oneside,reqno]{amsart}

\usepackage{amssymb,latexsym,amsxtra,amscd,ifthen}

\usepackage{amsmath}

\usepackage{amsfonts}

\usepackage{amsthm}

\usepackage{verbatim}

\usepackage{dsfont}

\usepackage{mathtools}

\usepackage{enumerate}

\usepackage[capitalise]{cleveref}

\usepackage{color}

\usepackage{endnotes}

\usepackage{tikz-cd}

\newtheorem{innercustomgeneric}{\customgenericname}
\providecommand{\customgenericname}{}
\newcommand{\newcustomtheorem}[2]{%
  \newenvironment{#1}[1]
  {%
   \renewcommand\customgenericname{#2}%
   \renewcommand\theinnercustomgeneric{##1}%
   \innercustomgeneric
  }
  {\endinnercustomgeneric}
}

\theoremstyle{plain}

\newtheorem{theorem}{Theorem}[section]

\newtheorem{lemma}[theorem]{Lemma}

\newtheorem{proposition}[theorem]{Proposition}

\newtheorem{corollary}[theorem]{Corollary}

\newtheorem*{corollary*}{Corollary}

\newtheorem*{theorem*}{Theorem}

\newcustomtheorem{customthm}{Theorem}

\newtheorem*{proposition*}{Theorem}

\theoremstyle{definition}

\theoremstyle{remark}

\DeclareMathAlphabet{\mathbbold}{U}{bbold}{m}{n}
\def\bb1{\mathbbold{1}}

\def\bcal{\mathcal{B}}

\DeclareMathOperator\ad{ad}

\DeclareMathOperator\rhs{RHS}

\newcommand{\til}[1]{\widetilde{#1}}

\newtheorem*{remark*}{Remark}

\usepackage{xparse}
\DeclareDocumentCommand{\adx}{ O{2} O{x_1}  }{\ad_{#2}^{[#1]}}

\begin{document}
\title{Matrix wreath products of algebras and embedding theorems}
\author{Adel Alahmadi\protect\endnotemark[1], Hamed Alsulami\protect\endnotemark[1], S.K. Jain\protect\endnotemark[1]$^,$\protect\endnotemark[2], Efim Zelmanov\protect\endnotemark[3]}
	
	

\address{1. To whom correspondence should be addressed\hfill\break
E-mail: ezelmano@math.ucsd.edu\hfill\break
Author Contributions: A.A., H.A, S. K. J., E. Z. designed research, performed research, and wrote the paper. The authors declare no conflict of interest.}
	
\begin{abstract}
We introduce a new construction of matrix wreath products of algebras that is similar to wreath products of groups. We then use it to prove embedding theorems for Jacobson radical, nil, and primitive algebras. In \S\ref{Section6}, we construct finitely generated nil algebras of arbitrary Gelfand-Kirillov dimension $\geq 8$ over a countable field which answers a question from \cite{8}.
\end{abstract}

\maketitle

\section{Main Results}\label{Section1}

G. Higman, H. Neumann, and B.H. Neumann \cite{13} proved that every countable group embeds in a finitely generated group. The papers \cite{3}, \cite{22}, \cite{23}, \cite{24}, \cite{25} show that some important properties can be inherited by these embeddings. Much of this work relies on wreath products of groups.

Following \cite{13}, A. I. Malcev \cite{20} showed that every countable dimensional associative algebra over a field is embeddable in a finitely generated algebra.

In \S\ref{Section2}, \ref{Section3}, we introduce matrix wreath products of algebras and study their basic properties.

In \S\ref{Section4}, we use matrix wreath products to prove embedding theorems for Jacobson radical algebras.

S. Amitsur \cite{2} asked if a finitely generated algebra can have a non nil Jacobson radical. The first examples of such algebras were constructed by K. Beidar \cite{5}. J. Bell \cite{6} constructed examples having finite Gelfand-Kirillov dimension. Finally, L. Bartholdi and A. Smoktunowicz \cite{4} constructed a finitely generated Jacobson radical non nil algebra of Gelfand-Kirillov dimension 2.

\begin{customthm}{4.1}\label{Theorem1}
An arbitrary countable dimensional Jacobson radical algebra is embeddable in a finitely generated Jacobson radical algebra.
\end{customthm}

\begin{customthm}{4.2}\label{Theorem2}
An arbitrary countable dimensional Jacobson radical algebra of Gelfand-Kirillov dimension $d$ over a countable field is embeddable in a finitely generated Jacobson radical algebra of Gelfand-Kirillov dimension $\leq d+6$.
\end{customthm}

We say that a nil algebra $A$ is \underline{stable nil} (resp. \underline{stable algebraic}) if all matrix algebras $M_n(A)$ are nil (resp. algebraic). The problem of Koethe (\cite{16}, see also \cite{17}) if all nil algebras are stable nil is still open.

\begin{customthm}{4.3}\label{Theorem3}
An arbitrary countable dimensional  stable nil algebra $A$ is embeddable in a finitely generated stable nil algebra. If $GK\dim A =d<\infty$ and the ground field is countable, then $A$ is embeddable in a finitely generated nil algebra of Gelfand-Kirillov dimension $\leq d+6$.
\end{customthm}

In \S\ref{Section5}, we prove embedding theorems for countable dimensional algebraic primitive algebras. I. Kaplansky \cite{15} asked if there exists an infinite dimensional finitely generated algebraic primitive algebra, a particular case of the celebrated Kurosh Problem. Such examples were constructed by J. Bell and L. Small in \cite{7}. Then J. Bell, L. Small, and A. Smoktunowicz \cite{8} constructed finitely generated algebraic primitive algebras of finite Gelfand-Kirillov dimension provided that the ground field is countable.

Our embedding theorems for algebraic primitive algebras have a special feature.

Let $A$ be an associative algebra over a ground field $F$. Let $X$ be a countable set. Consider the algebra $M_\infty(A)$ of $X\times X$ matrices over $A$ having finitely many nonzero entries. Clearly, the algebra $A$ is embeddable in $M_\infty(A)$ in many ways. We say that an algebra $A$ is \underline{$M_\infty$-embeddable} in an algebra $B$ if there exists an embedding $\varphi:M_\infty(A)\rightarrow B$. We say that $A$ is $M_\infty$-embeddable in $B$ as a (left, right) ideal if the image of $\varphi$ is a (left, right) ideal of $B$.

\begin{customthm}{5.1}\label{Theorem4}
An arbitrary countable dimensional stable algebraic primitive algebra is $M_\infty$-embeddable as a left ideal in a 2-generated algebraic primitive algebra.
\end{customthm}

In particular, this theorem answers the first part of question 7 from \cite{8}.

\begin{customthm}{5.2}\label{Theorem5}
Let $F$ be a countable field. An arbitrary countable dimensional stable algebraic primitive algebra of Gelfand-Kirillov dimension $\leq d$ is $M_\infty$-embeddable as a left ideal in a finitely generated algebraic primitive algebra of Gelfand-Kirillov dimension $\leq d+6$.
\end{customthm}

In \S\ref{Section6}, we answer question 1 from \cite{8}.

\begin{customthm}{6.1}\label{Theorem6}
Let $F$ be a countable field. For an arbitrary $d\geq 8$, there exists a finitely generated nil $F$-algebra of Gelfand-Kirillov dimension $d$.
\end{customthm}

\section{Matrix wreath products of algebras}\label{Section2}

Let $F$ be a field and let $A,B$ be two associative $F$-algebras. Let $Lin(A,B)$ denote the vector space of all $F$-linear transformations \linebreak$A\rightarrow B$.

We will define multiplication on $Lin(B, B\otimes_FA)$. Let \linebreak$f,g\in Lin(B,B\otimes_FA)$. For an arbitrary element $b\in B$, let \linebreak$g(b)=\sum\limits_ib_i\otimes a_i$, where $a_i\in A$, $b_i\in B$. Let $f(b_i)=\sum\limits_j b_{ij}\otimes a_{ij}$, where $a_{ij}\in A$, $b_{ij}\in B$. Define
\[(fg)(b)=\sum\limits_{i,j}b_{ij}\otimes a_{ij}a_i.\]
In other words, if $\mu:A\otimes A\rightarrow A$ is the multiplication on $A$, then
\[fg=(1\otimes\mu)(f\otimes 1)g.\]
Choose an arbitrary basis $\{b_i\}_{i\in I}$ of the algebra $B$ and a linear transformation $f:B\rightarrow B\otimes_F A$. Let
\[f(b_j)=\sum\limits_ib_i\otimes a_{ij}.\]
Consider the $I\times I$ matrix
\[A_f=(a_{ij})_{I\times I}.\]
Each column of this matrix contains only finitely many nonzero entries $a_{ij}$.

Let $f,g\in Lin(B,B\otimes_FA)$. Let $f(b_j)=\sum\limits_ib_i\otimes a_{ij}$, $g(b_j)=\sum\limits_ib_i\otimes a_{ij}'$, so $A_f=(a_{ij})_{I\times I}$, $A_g=(a_{ij}')_{I\times I}$. Then
\begin{align*}
(fg)(b_j)&=(1\otimes\mu)(f\otimes1)(\sum\limits_ib_i\otimes a_{ij}')\\
&=(1\otimes\mu)\sum\limits_{k,i}b_k\otimes a_{ki}\otimes a_{ij}'\\
&=\sum\limits_kb_k\otimes\sum\limits_ia_{ki}a_{ij}',
\end{align*}
which implies that
\[A_{fg}=A_fA_g.\]

Let $\til{M}_{I\times I}(A)$ denote the algebra of $I \times I$ matrices over $A$ having finitely many nonzero entries in each column. We proved that every basis of the algebra $B$ gives rise to an isomorphism
\[Lin(B,B\otimes_FA)\cong\til{M}_{I\times I}(A).\]
Let's define a structure of a $B$-bimodule on $Lin(B,B\otimes_FA)$. For an arbitrary element $b\in B$ and a linear transformation $f:B\rightarrow B\otimes_FA$, we will define linear transformations $fb$ and $bf$ via:
\begin{align*}
(fb)(b')=f(bb'), \hspace{0.2cm} b'\in B, \text{ and}\\
(bf)(b')=(b\otimes1)f(b').
\end{align*}
In other words, if $f(b')=\sum\limits_ib_i\otimes a_i$, then $(bf)(b')=\sum\limits_ibb_i\otimes a_i$. We will check that this is indeed a $B$-bimodule.

Choose arbitrary elements $b_1, b_2\in B$. Then
\[((fb_1)b_2)(b)=(fb_1)(b_2b)=f(b_1b_2b)=(f(b_1b_2))(b),\]
hence $(fb_1)b_2=f(b_1b_2)$. Also,
\begin{align*}
(b_1(b_2f))(b)&=(b_1\otimes1)(b_2f(b))=(b_1\otimes1)(b_2\otimes1)f(b)=(b_1b_2\otimes1)f(b)\\
&=((b_1b_2)f)b,
\end{align*}
hence $b_1(b_2f)=(b_1b_2)f$. Finally,
\[((b_1f)b_2)(b)=(b_1f)(b_2b)=(b_1\otimes1)f(b_2b).\]
On the other hand,
\[(b_1(fb_2))(b)=(b_1\otimes1)(fb_2)(b)=(b_1\otimes1)f(b_2b).\]
Hence, $(b_1f)b_2=b_1(fb_2)$.

Now, consider the semidirect sum
\[A\wr B=B+Lin(B,B\otimes_FA)\]
that extends multiplications on $B$ and on $Lin(B,B\otimes_FA)$.

\begin{proposition*}\label{Proposition1}
$A\wr B$ is an associative algebra.
\end{proposition*}
\begin{proof}
From the isomorphism $Lin(B,B\otimes_FA)\cong\til{M}_{I\times I}(A)$, we conclude that the algebra $Lin(B,B\otimes_FA)$ is associative. We checked above that $Lin(B,B\otimes_FA)$ is a bimodule over the associative algebra $B$. Hence, it remains to check that for arbitrary elements $b'\in B$; \linebreak$f,g\in Lin(B,B\otimes_FA)$, we have $(fb')g=f(b'g)$, $f(gb')=(fg)b'$, $(b'f)g=b'(fg)$. Indeed, let $b\in B$ and let $g(b)=\sum\limits_ib_i\otimes a_i$. Then
\[(fb'\otimes1)g(b)=\sum(fb')(b_i)\otimes a_i=\sum f(b'b_i)\otimes a_i.\]
Therefore,
\[((fb')g)(b)=(1\otimes\mu)\sum\limits_if(b'b_i)\otimes a_i.\]
Now consider the element $(f(b'g))(b)$. We have
\[(b'g)(b)=(b'\otimes1)g(b)=\sum\limits_ib'b_i\otimes a_i.\]
Applying $f\otimes1$, we get
\[(f\otimes1)(\sum\limits_ib'b_i\otimes a_i)=\sum\limits_if(b'b_i)\otimes a_i.\]
Finally,
\[(f(b'g))(b)=(1\otimes\mu)(f\otimes1)(b'g)(b)=(1\otimes\mu)\sum\limits_if(b'b_i)\otimes a_i.\]
Hence, $(fb')g=f(b'g)$.

Next,
\[((fg)b')(b)=(fg)(b'b)=(1\otimes\mu)(f\otimes1)g(b'b).\]
On the other hand,
\[(f(gb'))(b)=(1\otimes\mu)(f\otimes1)(gb')(b).\]
Now, $g(b'b)=(gb')(b)$ shows that $(fg)b'=f(gb')$. We will show that $(b'f)g=b'(fg)$. We have
\[((b'f)g)(b)=(1\otimes\mu)(b'f\otimes1)g(b)=(1\otimes\mu)(b'\otimes1\otimes1)(f\otimes1)g(b),\]
whereas
\[(b'(fg))(b)=(b'\otimes1)(1\otimes\mu)(f\otimes1)g(b).\]
Now it remains to notice that $(1\otimes\mu)(b'\otimes1\otimes1)=(b'\otimes1)(1\otimes\mu)$, which completes the proof of the proposition.
\end{proof}

We call $A \wr B=B+Lin(B,B\otimes_FA)$ the \underline{matrix wreath product} of the algebras $A,B$.

We remark that the above construction was preceded and inspired by
\begin{enumerate}[(i)]
\item constructions of examples in the paper \cite{8} by J. Bell, L. Small, and A. Smoktunowicz, and
\item a different definition of wreath products by Leavitt path algebras in the paper \cite{1} by A. Alahmadi and H. Alsulami.
\end{enumerate}

If $_BM$ is a left module over the algebra $B$, then we can define
\[A\wr_MB=B+Lin(M,M\otimes_FA).\]

\section{Properties of matrix wreath products}\label{Section3}

Fix an element $b\in B$. For a linear transformation $\gamma:B\rightarrow A$, consider an element $c_\gamma\in Lin(B,B\otimes_FA)$, $c_\gamma(b')=b\otimes\gamma(b')$ for an arbitrary element $b'\in B$. Clearly, $\rho_b=\{c_\gamma|\gamma\in Lin(B,A)\}$ is a right ideal of the algebra $Lin(B,B\otimes_FA)$.

Consider the algebra
\[S(A,B)=\sum\limits_{b\in B}\rho_b\vartriangleleft_rLin(B,B\otimes_FA).\]
The algebra $S(A,B)$ consists of linear transformations $\varphi:B\rightarrow B\otimes_FA$ such that there exists a finite dimensional subspace $V\subset B$ with \linebreak$\varphi(B)\subseteq V\otimes_FA$. Once we choose a basis $\{b_i, i\in I\}$ of the algebra $B$ and thus define an isomorphism $Lin(B,B\otimes_FA)\cong\til{M}_{I\times I}(A)$, the algebra $S(A,B)$ consists of (infinite) $I\times I$ matrices having finitely many nonzero rows. Recall that the algebra $M_\infty(A)$ consists of $I\times I$ matrices having finitely many nonzero entries.

In this section, we will study ring theoretic properties of $S(A,B)$ and of subalgebras $B+S<A\wr B$, where $M_\infty(A)\subseteq S\subseteq S(A,B)$. First, we will determine conditions for $B+S$ to be prime. Recall that an algebra is said to be prime if the product of any two nonzero ideals is not equal to zero.

In what follows, we assume that the algebra $B$ does not contain nonzero element $b$ such that $\dim_FbB<\infty$.

Let $\hat{A}$ denote the unital hull of the algebra $A$, i.e., $\hat{A}=A$ if $A$ contains 1, otherwise $\hat{A}=A+F\cdot1$.

For an element $b\in B$, let $L_b$ denote the operator of left multiplication $L_b:B\rightarrow B$, $x\rightarrow bx$. The operator $L_b$ can be viewed as a mapping $L_b:B\rightarrow B\otimes1$, hence $L_b\in Lin(B,B\otimes_F\hat{A})$. Denote
\[L_B=\{L_b,b\in B\}<Lin(B,B\otimes_F\hat{A}).\]

\begin{lemma}\label{Lemma1}
\begin{enumerate}[(1)]
\item \label{Item1} $L_BS(A,B)+S(A,B)L_B\subseteq S(A,B)$,
\item \label{Item2} $L_B\cap S(A,B)=(0).$
\end{enumerate}
\end{lemma}
\begin{proof}
Let $\varphi\in S(A,B)$, let $V\subset B$ be a finite dimensional subspace such that $\varphi(B)\subseteq V\otimes A$. Let $b\in B$. Then $(L_b\varphi)(B)\subseteq bV\otimes A$, $(\varphi L_b)(B)\subseteq \varphi(B)\subseteq V\otimes A$, which proves (\ref{Item1}).

If $b\in B$ and $L_b\in S(A,B)$, then $\dim_FbB<\infty$. By our assumption, it implies that $b=0$. This completes the proof of the lemma.
\end{proof}

Let $M_\infty(A)\subseteq S\subseteq S(A,B)$ be a subalgebra such that $BS+SB\subseteq S$.

\begin{proposition}\label{Proposition2}
The algebra $B+S$ is prime if and only if the algebra $A$ is prime.
\end{proposition}
\begin{proof}
Suppose that the algebra $B+S$ is prime. If $J_1,J_2$ are nonzero ideals of $A$, then $M_\infty(J_1), M_\infty(J_2)$ are nonzero left ideals of the algebra $B+S$. If $J_1J_2=(0)$, then $M_\infty(J_1)M_\infty(J_2)=(0)$. It is well known that a product of two nonzero left ideals in a prime algebra is not equal to zero. That contradicts the primeness of $B+S$.

Suppose now that the algebra $A$ is prime. Then the algebra $M_\infty(A)$ is prime as well. If $K_1, K_2$ are nonzero ideal of $B+S$ such that \linebreak$K_1K_2=(0)$, then $K_1\cap M_\infty(A)=(0)$ or $K_2\cap M_\infty(A)=(0)$. Since $M_\infty(A)$ is a left ideal of $B+S$, it follows that $K_1M_\infty(A)=(0)$ or $K_2M_\infty(A)=(0)$. Hence, $M_\infty(A)$ has a nonzero left annihilator in $B+S$. Let $0\neq b\in B$, $s\in S$, and suppose that $(b+s)M_\infty(A)=(0)$. Since the algebra $M_\infty(A)$ has zero left annihilator in $\til{M}_{I\times I}(A)$, it follows that $\til{M}_{I\times I}(A)(b+s)=(0)$.

For an arbitrary element $f\in Lin(B,B\otimes_FA)$, we have \linebreak$(fb)(b')=f(bb')=(fL_b)(b')$. Hence, $L_b+s=0$. By Lemma \ref{Lemma1} (\ref{Item1}), $b=0$ and it remains to recall again that $M_\infty(A)$ has zero left (right) annihilators in $\til{M}_{I\times I}(A)$. This completes the proof of the proposition.
\end{proof}

Next we will find conditions for $B+S$ to be primitive. Recall that an algebra is said to be (left) primitive if it has a faithful irreducible left module \cite{14}.

\begin{lemma}\label{Lemma2}
Let $R$ be a prime algebra with a nonzero left ideal $L\vartriangleleft_eR$. Suppose that
\begin{enumerate}[(1)]
\item \label{Assumption1} $\{\ell\in L|L\ell=(0)\}=(0)$,
\item \label{Assumption2} for arbitrary $n\geq 1$; arbitrary elements $a\in R$; and $\ell_1, \cdots, \ell_n\in L$, there exists an element $\ell'\in L$ such that $(a-\ell')\ell_i=0$, $1\leq i\leq n$.
\end{enumerate}
Then the algebra $R$ is primitive if and only if the algebra $L$ is primitive.
\end{lemma}
\begin{proof}
Let $M$ be a faithful irreducible left module over $R$. Consider the subspace $M'=\{m\in M|Lm=(0)\}$. Because of faithfulness of $M$, we have $M'\lvertneqq M$. We will show that the factor space $M/M'$ is a faithful irreducible $L$-module.

Indeed, if $\ell\in L$ and $\ell(M/M')=(0)$, then $L\ell M=(0)$, which implies that $L\ell=(0)$. From (\ref{Assumption1}), we conclude that $\ell=0$. We will show that the $L$-module $M/M'$ is irreducible. Let $0\neq m+M'\in M/M'$. Then $Lm=M$, which implies $L(m+M')=M/M'$.

Now suppose that the algebra $L$ is primitive and M is a faithful irreducible left module over $L$. We will define a structure of an $R$-module on $M$. Since $LM=M$, an arbitrary element of $M$ can be represented as $\sum\limits_{i=1}^n\ell_im_i$, $\ell_i\in L$, $m_i\in M$. For an element $a \in R$, we define $a(\sum\limits_{i=1}^n\ell_im_i)=\sum\limits_{i=1}^n(a\ell_i)m_i$. To check that this action is well defined, we have to show that $\sum\limits_{i=1}^n\ell_im_i=0$ implies $\sum\limits_{i=1}^n(a\ell_i)m_i=0$. By assumption (\ref{Assumption2}), these exists an element $\ell'\in L$ such that $a\ell_i=\ell'\ell_i$, $1\leq i\leq n$. Hence, $\sum\limits_{i=1}^n(a\ell_i)m_i=\sum\limits_{i=1}^n\ell'\ell_im_i=0$. This completes the proof of the lemma.
\end{proof}

\begin{proposition}\label{Proposition3}
The algebra $B+S$ is primitive if and only if the algebra $A$ is primitive.
\end{proposition}
\begin{proof}
Suppose that the algebra $B+S$ is primitive. Since a nonzero two sided ideal of a primitive algebra is primitive, we conclude that the algebra $S$ is primitive and therefore prime. The algebra $A$ is also prime by Proposition \ref{Proposition2}.

We will check whether the algebra $S$ and its left ideal $M_\infty(A)$ satisfy assumptions (\ref{Assumption1}), (\ref{Assumption2}) of Lemma \ref{Lemma2}. Part (\ref{Assumption1}) is trivial. Now we will check assumption (\ref{Assumption2}). Choose elements $a \in S$; $a_1,\cdots, a_n\in M_\infty(A)$. Let $b_1, \cdots, b_m$ be elements of the basis of the algebra $B$ such that $a_1, \cdots, a_n$ have only nonzero rows that correspond to $b_1, \cdots, b_m$. In other words, $a_1, \cdots, a_m \in \sum\limits_{i=1}^m\rho_{b_i}$.

Let $a'$ be the $I\times I$ matrix that has the same entries as $a$ in the columns that correspond to $b_1, \cdots, b_m$ and zeros everywhere else. Then $a'\in M_\infty(A)$ and $aa_i=a'a_i$, $1\leq i \leq n$. By Lemma \ref{Lemma2}, the left ideal $M_\infty(A)$ of the algebra $S$ is a primitive algebra, which implies primitivity  of the algebra $A$.

Now suppose that the algebra $A$ is primitive. By Proposition \ref{Proposition2}, the algebra $B+S$ and $S$ are prime. By Lemma \ref{Lemma2}, the algebra $S$ is primitive. It is easy to see that if a nonzero ideal of a prime algebra is primitive, then the full algebra is primitive as well. This finishes the proof of the proposition.
\end{proof}

In the rest of this section, we will study growth of some subalgebras in $A \wr B$. We will recall some definitions.

Let $R$ be an $F$-algebra generated by a finite dimensional subspace $V$. Let
\[V^n=\text{span}_F(v_1 \cdots v_k|k\leq n, v_i\in V, 1\leq i\leq k).\]
Then $\dim_FV^n< \infty$ and $R$ is the union of the ascending chain\linebreak $V^1\subseteq V^2\subseteq \cdots$. The function $g(V,n)=\dim_FV^n$ is called the growth function of the algebra $R$ that corresponds to the generating subspace $V$.

Given two functions $f_1, f_2:N\rightarrow [1,\infty)$, we say that $f_1$ is asymptotically less than or equal to $f_2$ (denote: $f_1 \preceq f_2$) if there exists $c\in N$ such that $f_1(n)\leq cf_2(cn)$ for all $n$. If $f_1\preceq f_2$ and $f_2 \preceq f_1$, then we say that $f_1$ and $f_2$ are asymptotically equivalent (denote: $f_1 \sim f_2$).

If $V_1, V_2$ are two finite dimensional generating subspaces of $R$, then $g(V_1, n)\sim g(V_2,n)$. We will denote the class of functions that are equivalent to $g(V,n)$ as $g_R(n)$.

If there exists $\alpha>0$ such that $g_R(n)\preceq n^\alpha$, then we say that growth of $R$ is polynomially bounded. In this case
\[GK\dim(R)=\inf\{\alpha>0|g_R(n)\preceq n^\alpha\}\]
is called the Gelfand-Kirillov dimension of $R$. If $R$ does not have polynomially bounded growth, then $GK\dim(R)=\infty$.

For a not necessarily finitely generated algebra $R$, we let
\[GK\dim(R)=\sup GK\dim(R'),\]
where $R'$ runs over all finitely generated subalgebras of $R$.

Coming back to the algebras $A,B$, we say that a linear transformation $\gamma: B\rightarrow A$ is a generating linear transformation if $\gamma(B)$ generates $A$.

Now suppose that the algebra $B$ contains 1. Let $\gamma:B\rightarrow A$ be a generating linear transformation. As above, we consider the element $c_\gamma:b\rightarrow1\otimes\gamma(b)\in B\otimes_FA$.

If $a\in A$, then we denote $ac_\gamma=c_{\gamma'}$, where $\gamma'(b)=a\gamma(b)$.

Consider the subalgebra $C=\langle B,c_\gamma\rangle$ generated in $A\wr B$ by $B$ and the element $c_\gamma$.

If $V$ is a generating subspace of the algebra $B$, then $U=V+Fc_\gamma$ is a generating subspace of the algebra $C$.

For $n\geq 1$, consider the vector space
\[W_n=\sum\limits_{i_1+\cdots+i_r\leq n}\gamma(V^{i_1})\cdots\gamma(V^{i_r})\subseteq A.\]
Clearly, $W_1\subseteq W_2\subseteq\cdots$, $A=\bigcup_{n\geq1}W_n$.

\begin{lemma}\label{Lemma3}
$U^n\subseteq \sum\limits_{i+j+k\leq n} V^i(W_jc_\gamma)V^k+V^n$ for any $n\geq1$.
\end{lemma}
\begin{proof}
For $n=1$, the assertion is obvious. We denote the right hand side of the inclusion above as $\rhs(n)$. We need to show that $U\rhs(n-1)\subseteq \rhs(n)$. Clearly, $V\rhs(n-1)\subseteq \rhs(n)$. Let \linebreak$v\in V^i$. Then $c_\gamma v=c_{\gamma'}$, where $\gamma'(b)=\gamma(vb)$. We have \linebreak$\gamma'(1)=\gamma(v)\in W_i$. Now,
\[c_\gamma v(W_jc_\gamma)=c_\gamma'(W_jc_\gamma)=(\gamma'(1)W_j)c_\gamma\subseteq W_{i+j}c_\gamma.\]
Therefore,
\[c_\gamma V^i(W_jc_\gamma)V^k\subseteq(W_{i+j}c_\gamma)V^k\subseteq \rhs(n-1)\subseteq \rhs(n).\]
This completes the proof of the lemma.
\end{proof}

Denote $w_\gamma(n)=\dim_FW_n$.

\begin{corollary}\label{Corollary1}
$g_C(n)\preceq g_B^2(n)w_\gamma(n)$.
\end{corollary}

If $A\ni1$, then along with the algebra $C$, we will consider a bigger algebra $C'=\langle B,c_\gamma, e_{11}(1)\rangle$ and its generating subspace
\[U'=V+Fc_\gamma+Fe_{11}(1).\]

\begin{lemma}\label{Lemma4}
\[U'^n\subseteq\sum\limits_{i+j+k\leq n}V^i(W_jc_\gamma)V^k+V^n+\sum\limits_{i+j+k\leq n}V^ie_{11}(W_j)V^k.\]
\end{lemma}
\begin{proof}
Again, denote the right hand side of the inclusion as $\rhs(n)$. We need to check that $c_\gamma\sum\limits_{i+j+k\leq n-1}V^ie_{11}(W_j)V^k\subseteq \rhs(n)$ and \linebreak$e_{11}(1)\rhs(n-1)\subseteq \rhs(n)$. The subspace $e_{11}(W_j)$ lies in $\rho_1$. Hence,
\[c_\gamma V^ie_{11}(W_j)\subseteq e_{11}(\gamma(V^i)W_j)\subseteq e_{11}(W_{i+j})\]
and therefore $e_{11}(W_{i+j})V^k\subseteq \rhs(n-1)$. Furthermore,
\[e_{11}(1)V^i(W_jc_\gamma)V^k=e_{11}(1)V^ie_{11}(1)(W_jc_\gamma)V^k\]
and it remains to notice that $e_{11}(1)V^ie_{11}(1)=Fe_{11}(1)$. This completes the proof of the lemma.
\end{proof}

\begin{corollary}\label{Corollary2}
$g_{C'}(n)\preceq g_B^2(n)w_\gamma(n)$.
\end{corollary}

We say that a linear transformation $\gamma:B\rightarrow A$ is \underline{dense} if for arbitrary linearly independent elements $b_1,\cdots, b_n\in B$ and arbitrary nonzero element $a\in A$, there exists an element $b\in B$ such that $\gamma(b_ib)=0$, $1\leq i\leq n-1$, and $a\gamma(b_nb)\neq0$.

\begin{lemma}\label{Lemma4}
If $\gamma:B\rightarrow A$ is a dense generating linear transformation, then $g_C(n)\sim g_B(n)^2w_\gamma(n)$.
\end{lemma}
\begin{proof}
It is easy to see that
\[V^n(W_nc_\gamma)V^n\subseteq U^{3n}.\]
We will show that $\dim_FV^n(W_nc_\gamma)V^n=(\dim_FV^n)^2w(n)$. Let $b_1,\cdots, b_r$ be a basis of $V^n$ and let $a_1,\cdots, a_t$ be a basis of $W_n$. We need to verify that elements $b_i(a_jc_\gamma)b_k$ are linearly independent.

For an arbitrary element $b\in B$ and arbitrary coefficients $\gamma_{ijk}\in F$, we have
\[\Big(\sum\gamma_{ijk}b_i(a_jc_\gamma)b_k\Big)(b)=\sum\gamma_{ijk}b_i\otimes a_j\gamma(b_kb).\]
Since the elements $b_i$ are linearly independent, it follows that for every $i$,
\[\sum\limits_{j,k}\gamma_{ijk}a_j\gamma(b_kb)=0.\]

Let $\gamma_{i_0j_0k_0}\neq 0$. By density of $\gamma$, there exists an element $b\in B$ such that $\gamma(b_\ell b)=0$ for $\ell\neq k_0$ and $(\sum\gamma_{i_0jk_0}a_j)\gamma(b_{k_0}b)\neq 0$, a contradiction.
\end{proof}

\begin{lemma}\label{Lemma5}
Suppose that the algebra $B$ has a basis $b_1, b_2, \cdots$ that consists of invertible elements. Suppose that $A\ni1$. The basis $\{b_i\}_{i\in I}$ defines the isomorphism $Lin(B, B\otimes_FA)\cong\til{M}_{I\times I}(A)$. Let $\gamma: B\to A $ be a generating linear transformation. Then the algebra $C'=\langle B,c_\gamma,e_{11}(1)\rangle$ contains $M_\infty(A)$.
\end{lemma}
\begin{proof}
We have $e_{ij}(1)=b_ie_{11}(1)b_j^{-1}$, hence $C'\supseteq M_\infty(F)$. If $\gamma(b_i)=a$, then $c_\gamma be_{11}(1)=e_{11}(a)$. Since $\gamma$ is a generating linear transformation, it follows that $C'\supseteq e_{11}(A)$. Now it remains to notice that $M_\infty(F)$ and $e_{11}(A)$ generate $M_\infty(A)$.
\end{proof}

\section{Radical Algebras}\label{Section4}

In this section, we will prove embedding theorems \ref{Theorem1}-\ref{Theorem3} for Jacobson radical algebras.

\begin{lemma}\label{Lemma6}
For an arbitrary Jacobson radical algebra $A$, there exists a Jacobson radical algebra $\til{A}$ and an element $u\in\til{A}$, $u^3=0$, such that $A$ is embeddable in the right ideal $u\til{A}$ (resp. left ideal $\til{A}u$).
\end{lemma}
\begin{proof}
Consider the two dimensional nilpotent algebra $B$ with a basis $b_1=b$, $b_2=b^2$, $b^3=0$. Let $A$ be a Jacobson radical algebra. Consider the matrix wreath product $A \wr B=B+M_2(A)$. Clearly, $A\wr B$ is a Jacobson radical algebra. For $1\leq i,j\leq 2$ and an element $a\in A$, we consider the linear transformation $e_{ij}(a)$ that maps a basic element $b_k$ to $\delta_{ik}b_j\otimes a$. Then $be_{21}(a)=e_{22}(a)$. Hence $e_{22}(A)\subseteq b(A\wr B)$, which completes the proof of the lemma.
\end{proof}

\begin{proof}[Proof of Theorem \ref{Theorem1}]
Let $A$ be a countable dimensional Jacobson radical algebra. By Lemma \ref{Lemma6}, there exists a countable dimensional Jacobson radical algebra $\til{A}$ and an element $u\in\til{A}$, $u^3=0$, such that $A$ embeds in $\til{A}u$.

Let $B$ be a finitely generated infinite dimensional nil algebra of E. S. Golod \cite{11}. Let $\hat{B}=B+F\cdot1$ be its unital hull. Let $\gamma:\hat{B}\rightarrow\til{A}$ be a generating linear transformation. In the matrix wreath product $\til{A}\wr\hat{B}$, consider the element $c_\gamma:\hat{B}\rightarrow\hat{B}\otimes_F\til{A}$, $c_\gamma(b)=1\otimes\gamma(b)$.

Choose a basis $\{b_i\}_{i\in I}$ of the algebra $\hat{B}$, $b_1=1$. It gives rise to an isomorphism $Lin(\hat{B},\hat{B}\otimes_F\til{A})\rightarrow\til{M}_{I\times I}(\til{A})$. Consider the element $e_{11}(u)\in Lin(\hat{B},\hat{B}\otimes_F\til{A})$ that sends $b_1=1$ to $1\otimes u$ and sends other basic elements to zero. Consider the subalgebra $C$ of $\til{A}\wr\hat{B}$ generated by $B, C_\gamma, e_{11}(u)$.

Consider also the right ideal
\[\rho_1=\{c_\alpha | \alpha \in Lin(\hat{B},\til{A}), c_\alpha(b)=1\otimes\alpha(b)\}.\]
For any $\alpha,\beta\in Lin(\hat{B},\til{A})$, we have $c_\alpha c_\beta=\alpha(1)c_\beta$. Hence, the mapping $\pi:\rho_1\rightarrow\til{A}$, $\pi(c_\alpha)=\alpha(1)$ is a homomorphism.

The subalgebra $\langle c_\gamma\hat{B}\rangle$ generated by $c_\gamma\hat{B}$ lies in $\rho_1$. For any basic element $b_i$, we have $\pi(c_\gamma b_i)=\gamma(b_i)$. Since $\gamma$ is a generating linear transformation, it follows that the restriction of $\pi$ to $\langle c_\gamma B\rangle$ is surjective. Hence
\[C\supseteq\langle c_\gamma\hat{B}\rangle e_{11}(u)=e_{11}(\til{A}u)\supseteq e_{11}(A).\]

It remains to show that the subalgebra $C$ is Jacobson radical. We will start by showing that the right ideal $Fc_\gamma+c_\gamma C$ is Jacobson radical. The right ideal $Fc_\gamma+c_\gamma C$ is contained in $\rho_1$ and contains $\langle c_\gamma\hat{B}\rangle$. Hence, the restriction of the homomorphism $\pi$ to $Fc_\gamma+c_\gamma C$ is surjective. The kernel of this homomorphism lies in
\[\rho_1'=\{c_\alpha|\alpha(1)=0\},\]
with $(\rho_1')^2=(0).$ This proves that the right ideal $Fc_\gamma+c_\gamma C$ of the algebra $C$ is Jacobson radical. Hence, $c_\gamma$ lies in the Jacobson radical $Jac(C)$ of the algebra $C$.

The ideal generated by $e_{11}(u)$ in the subalgebra $\langle B,e_{11}(u)\rangle$ lies in $M_{I\times I}(uF[u])$, hence this ideal is nilpotent. Hence $e_{11}(u)\in Jac(C)$.

Finally, it follows that $C/Jac(C)=B+Jac(C)/Jac(C)$, a nil algebra, which implies that $C=Jac(C)$. The algebra $C$ is finitely generated. This completes the proof of Theorem \ref{Theorem1}.
\end{proof}

Now we turn to Theorem \ref{Theorem2}. Let $A$ be a countable dimensional algebra of Gelfand-Kirillov dimension $\leq d$. Let the algebra $B$ be generated by a finite dimensional subspace $V$. Recall that for a linear transformation $\gamma:B\rightarrow A$, we denote
\[W_n=\sum\limits_{i_1+\cdots+i_r\leq n}\gamma(V^{i_1})\cdots\gamma(V^{i_r}), \hspace{1cm} w_\gamma(n)=\dim_FW_n.\]

\begin{lemma}\label{Lemma7}
There exists a generating linear transformation \linebreak$\gamma:B\rightarrow A$ such that $w_\gamma(n)\leq n^{d+\epsilon_n}$, where $\epsilon_n>0$, $\epsilon_n\rightarrow0$ as $n\rightarrow\infty$.
\end{lemma}
\begin{proof}
Let $a_1, a_2, \cdots$ be a basis of the algebra $A$. Let
\[g_k(n)=\dim_F \text{span}(a_{i_1}\cdots a_{i_r}; 1\leq r\leq n; 1\leq i_1,\cdots, i_r\leq k).\]
From $GK\dim A\leq d$, it follows that there exists an increasing sequence $n_k$, $k\geq 1$, such that $g_k(n)\leq n^{d+\frac{1}{k}}$ as soon as $n\geq n_k$. For $n\geq n_1$, choose $k$ such that $n_k\leq n< n_{k+1}$. Let $\epsilon_n=\frac{1}{k}$. It is clear that $\epsilon_n\rightarrow0$ as $n\rightarrow\infty$. Choose a subspace $V_k'\subset V^{n_k}$ and an element $v_k\in V^{n_k}$ such that $V^{n_k}=V^{n_{k-1}}\oplus V_k'\oplus Fv_k$ is a direct sum of subspaces. Then $B=V_1'\oplus Fv_1\oplus V_2'\oplus Fv_2\oplus\cdots$. Define a linear transformation $\gamma: B\rightarrow A$ via $\gamma(V_i')=0$, $i\geq 1$, $\gamma(v_i)=a_i$.

The subspace $W_n$ is spanned by $\gamma(V^{i_1})\cdots\gamma(V^{i_r})$, $i_1+\cdots+i_r\leq n$. Hence, $\gamma(V^{i_1})\cdots\gamma(V^{i_n})\subseteq \text{span}_F(a_{j_1}\cdots a_{j_r}; 1\leq j_1,\cdots, j_r\leq k)$. Now we get $w_\gamma(n)\leq g_k(n)\leq n^{d+\frac{1}{k}}$ as $n\geq n_k$. This completes the proof of the lemma.
\end{proof}

\begin{proof}[Proof of Theorem \ref{Theorem2}]
Let $A$ be a countable dimensional Jacobson radical algebra of Gelfand-Kirillov dimension $\leq d$. Let the ground field $F$ be countable. In \cite{18}, T. Lenagan and A. Smoktunowicz constructed a finitely generated nil $F$-algebra of finite Gelfand-Kirillov dimension. In \cite{19}, T. Lenagan, A. Smoktunowicz, and A. Young refined the argument of \cite{18} to construct a finitely generated nil algebra $B$ of Gelfand-Kirillov dimension $\leq 3$.

Following Lemma \ref{Lemma7}, there exists a generating linear transformation $\gamma:\hat{B}\rightarrow\til{A}$ such that $w_\gamma(n)\leq n^{d+\epsilon_n}$, $\epsilon_n\rightarrow0$ as $n\rightarrow\infty$. As shown above, the algebra $A$ embeds in a finitely generated algebra $C'=\langle B,c_\gamma, e_{11}(u)\rangle$. By Corollary \ref{Corollary2}, $g_{c'}(n)\preceq g_B(n)^2w_\gamma(u)$. This implies $GK\dim C'\leq d+6$.  This completes the proof of Theorem \ref{Theorem2}.
\end{proof}

For the proof of Theorem \ref{Theorem3}, we need to recall more details about the Golod-Shafarevich inequality (see \cite{12}) and Golod's construction \cite{11}.

Let $F\langle x_1, \cdots, x_m\rangle$ be the free associative algebra on $m$ free generators, $m\geq 2$. We consider the free algebra without 1, i.e., it consists of formal linear combinations of nonempty words in $x_1, \cdots, x_m$. Assigning degree 1 to all variables $x_1, \cdots, x_m$, we make $F\langle x_1, \cdots, x_m\rangle$ a graded algebra. The degree $\deg(a)$ of an arbitrary element $a\in F\langle x_1, \cdots, x_m\rangle$ is defined as the minimal degree of a nonzero homogeneous component of $a$.

Let $R\subset F\langle x_1,\cdots, x_m\rangle$ be a subset containing finitely many elements of each degree.

\textit{Golod-Shafarevich Condition:} If there exists a number $0<t_0<1$ such that
\[\sum\limits_{a\in R}t_0^{\deg(a)}<\infty \text{\hspace{0.2cm} and\hspace{0.2cm}} 1-mt_0+\sum\limits_{a\in R}t_0^{\deg(a)}<0,\]
then the algebra $\langle x_1, \cdots, x_m|R=0\rangle$ presented by the set of generators $x_1, \cdots, x_m$ and the set of relations $R$ is infinite dimensional.

Recall that a function $g:N\rightarrow[1,\infty)$ is said to be subexponential if $\lim\limits_{n\rightarrow\infty}\dfrac{g(n)}{e^{\alpha n}}=0$ for any $\alpha>0$. A finitely generated algebra $A$ has subexponential growth if its growth function $g_A(n)$ is subexponential. It is equivalent to $g_A(n)\precneqq e^n$.

A (not necessarily finitely generated) algebra $A$ is of locally subexponential growth if every finitely generated subalgebra of $A$ is of subexponential growth.

\begin{lemma}\label{Lemma8}
Let $F$ be a countable field and let $A$ be a countable dimensional $F$-algebra of locally subexponential growth. Then there exists a subset $R\subset F\langle x_1,\cdots, x_m\rangle$ satisfying the Golod-Shafarevich condition and such that the algebra $F\langle x_1, \cdots, x_m|R=0\rangle \otimes_F A$ is nil.
\end{lemma}
\begin{proof}
The algebra $F\langle x_1,\cdots, x_m\rangle\otimes_FA$ is countable. Let
\[F\langle x_1,\cdots,x_m\rangle\otimes_FA=\{f_1,f_2,\cdots\}.\]
Choose $\frac{1}{m}<t_0<1$ and a sequence $\epsilon_1,\epsilon_2,\cdots>0$ such that $\sum\limits_{i=1}^\infty\epsilon_i<\infty$ and $1-mt_0+\sum\limits_{i=1}^\infty\epsilon_i<0$. Choose $i\geq1$. Let $f_i\in\sum\limits_{j=1}^kF\langle x_1,\cdots, x_m\rangle\otimes a_j$, \linebreak $a_j\in A$, $V_i=\sum\limits_jFa_j$. Then for an arbitrary $n\geq1$, we have \linebreak$f_i^n\in F\langle x_1,\cdots,x_m\rangle^n\otimes V_i^n$. Let $g_{V_i}(n)=\dim_FV_i^n$. Since the function $g_{V_i}(n)$ is subexponential, there exists $n_i\geq 1$ such that for all $n\geq n_i$, we have $g_{V_i}(n)t_0^n\leq\epsilon_i$. Let $r=g_{V_i}(n_i)$, let $v_{i1},\cdots, v_{ir}$ be a basis of $V_i^{n_i}$ and let $f_i^{n_i}=\sum\limits_{j=1}^rf_{ij}\otimes v_{ij}$, $\deg f_{ij}\geq n_i$.

Let $R=\{f_{ij}|i\geq1, 1\leq j\leq g_{V_i}(n_i)\}$. The image of an element $f_i$ in $F\langle x_1,\cdots, x_m|R=0\rangle\otimes_FA$ is nilpotent of index $\leq n_i$. Besides,
\vspace{3mm}
\[\sum g_{V_i}(n_i)t_0^{\deg(f_{ij})}\leq \sum g_{V_i}(n_i)t_0^{n_i}\leq \sum\epsilon_i.\vspace{3mm}\]
Hence, $R$ satisfies the Golod-Shafarevich Condition and therefore the algebra $F\langle x_1,\cdots, x_m|R=0\rangle\otimes_FA$ is an infinite dimensional nil algebra. This completes the proof of the lemma.
\end{proof}

\begin{lemma}\label{Lemma9}
Let $F$ be an arbitrary field. There exists an infinite dimensional finitely generated stable nil $F$-algebra.
\end{lemma}
\begin{proof}
Let $F_0$ be the prime subfield of $F$. We will apply Lemma \ref{Lemma8} to the countable dimensional $F_0$-algebra $A=F_0[t_i,i\geq1]\otimes_{F_0}M_\infty(F_0)$ of locally subexponential growth. By Lemma \ref{Lemma8}, there exists a subset $R\subset F_0\langle x_1,\cdots,x_m\rangle$ satisfying the Golod-Shafarevich condition such that the $F_0$-algebra $F_0\langle x_{1},\cdots,x_m|R=0\rangle\otimes_{F_0}A$ is nil. We will show that the $F$-algebra $F\langle x_1,\cdots, x_m|R=0\rangle$ is stable nil. Indeed, we need to check only that for arbitrary elements $a_1,\cdots, a_k\in F_0\langle x_1,\cdots, x_m|R=0\rangle$, arbitrary elements $\alpha_1,\cdots,\alpha_k\in F$ and arbitrary matrices $y_1,\cdots, y_k\in M_n(F_0)$, the tensor $\sum\limits_{i=1}^ka_i\otimes\alpha_i\otimes y_i$ is nilpotent. It follows from the nilpotency of the element $\sum\limits_{i=1}^ka_i\otimes t_i\otimes y_i$ of the algebra $F_0\langle x_1,\cdots, x_m|R=0\rangle\otimes_{F_0}A$. This completes the proof of the lemma.
\end{proof}

\begin{lemma}\label{Lemma10}
Let $A$ be a stable nil (resp. algebraic) algebra. Then for an arbitrary algebra $B$, the algebra $S(A,B)$ is nil (resp. algebraic).
\end{lemma}
\begin{proof}
Recall that $S(A,B)$ consists of $I\times I$ matrices over $A$ with finitely many nonzero rows. For a finite subset $T\subset I$, let $S_T(A,B)$ consist of such matrices that for any $i\in I\setminus T$, the $i^{th}$ row is zero. The algebra $S(A,B)$ is a union of subalgebras $S_T(A,B)$. Consider the mapping $S_T(A,B)\xrightarrow[]{\varphi}M_{T\times T}(A)$. For an arbitrary matrix $Y\in S_T(A,B)$, the matrix $\varphi(Y)$ is the part of $Y$ at the intersection of rows and columns indexed by $T$. Clearly, $\varphi$ is a homomorphism and $(\text{ker}\varphi)^2=(0)$. This implies that the algebra $S_T(A,B)$ is nil (resp. algebraic) and completes the proof of the lemma.
\end{proof}

\begin{proof}[Proof of Theorem \ref{Theorem3}]
Let $B$ be an infinite dimensional finitely generated stable nil algebra of Lemma \ref{Lemma9}. If $A$ is a countable dimensional stable nil algebra, then the algebra $\til{A}=A\wr F[u|u^3=0]$, $A\hookrightarrow\til{A}u$, is also countable dimensional and stable nil due to Lemma \ref{Lemma6}.

In the proof of Theorem \ref{Theorem1}, we embedded the algebra $A$ in a finitely generated subalgebra $C=\langle B,C_\gamma,e_{11}(u)\rangle$ of $\til{A}\wr\hat{B}$. Now we need to show that for an arbitrary $n\geq1$, the matrix algebra $M_n(C)$ is nil. It is easy to see that $M_n(C)$ is a subalgebra of $M_n(\til{A})\wr M_n(\hat{B})$. Moreover,
\[M_n(C)\subseteq M_n(B)+S(M_n(\til{A}),M_n(\hat{B})).\]
By Lemmas \ref{Lemma1} (\ref{Item1}), \ref{Lemma10}, $S(M_n(\til{A}),M_n(\hat{B}))$ is a nil ideal of the algebra $M_n(B)+S(M_n(\til{A}),M_n(\hat{B}))$. The algebra $M_n(B)$ is nil because the algebra $B$ is stable nil. We proved that the algebra $C$ is stable nil.

Suppose now that the algebra $A$ is stable nil, the ground field $F$ is countable, and $GK\dim A\leq d$. Let $B$ be the Lenagan-Smoktunowicz-Young (\cite{18}, \cite{19}) nil algebra, $GK\dim B\leq 3$. Arguing as above, we see that the finitely generated algebra $C$, in which the algebra $A$ is embedded, is nil. By Corollary \ref{Corollary1}, the growth of $C$ is bounded by $n^6w_\gamma(n)$. By Lemma \ref{Lemma7}, a generating linear transformation $\gamma$ can be chosen so that $w_\gamma(n)\leq n^{d+\epsilon_n}$, where $\epsilon_n\rightarrow0$ as $n\rightarrow\infty$. This implies that $GK\dim C\leq d+6$ and finishes the proof of the theorem.
\end{proof}

\begin{remark*}
If we knew that there exists a Lenagan-Smoktunowicz algebra that is stable nil, then we could embed a countable dimensional stable nil algebra of finite Gelfand-Kirillov dimension in a finitely generated stable nil algebra of finite Gelfand-Kirillov dimension.
\end{remark*}

\section{Algebraic Primitive Algebras}\label{Section5}

The purpose of this section is to prove Theorems \ref{Theorem4}, \ref{Theorem5}. Let $A$ be a countable dimensional stable algebraic primitive algebra. Without loss of generality, we will assume that $A\ni1$. Let $B$ be an infinite dimensional finitely generated stable nil algebra of Lemma \ref{Lemma9}. Without loss of generality, we will also assume that $\{b\in B|\dim_FbB<\infty\}=(0)$.

Consider the matrix wreath product $A\wr\hat{B}$ and an element \linebreak$c_\gamma\in Lin(\hat{B},\hat{B}\otimes_FA)$, $c_\gamma(b)=1\otimes\gamma(b)$, where $\gamma:\hat{B}\rightarrow A$ is a generating linear transformation. Choose a basis $\{b_i\}_{i\in N}$ of the algebra $B$ that consists of invertible elements. This basis defines an isomorphism
\[Lin(\hat{B},\hat{B}\otimes_FA)\cong\til{M}_{N\times N}(A).\]

As above, we consider the finitely generated algebra \linebreak$C'=\langle \hat{B}, c_\gamma, e_{11}(1)\rangle$. By Lemma \ref{Lemma5}, $M_\infty(A)\subseteq C'$. Since $M_\infty(A)$ is a left ideal in $\til{M}_{N\times N}(A)$ and $BM_\infty(A)\subseteq M_\infty(A)$, it follows that $M_\infty(A)$ is a left ideal in $C'$.

Arguing precisely as in the proof of Theorem \ref{Theorem3} and using Lemma \ref{Lemma10}, we can show that the algebra $C'$ is stable algebraic.

By Proposition \ref{Proposition3}, the algebra $C'$ is primitive.

By a theorem of V. T. Markov \cite{21}, there exists $n\geq1$ such that the matrix algebra $M_n(C')$ is 2-generated. Since $M_n(M_\infty(A))\cong M_\infty(A)$, the algebra is still $M_\infty$-embedded in $M_n(C')$ as a left ideal. The algebra $M_n(C')$ is still stable algebraic and primitive. This finishes the proof of Theorem \ref{Theorem4}.

Assume now that the ground field $F$ is countable, the algebra $A$ is stable algebraic and primitive, and $GK\dim A\leq d$. For the algebra $B$, we now take the Lenagan-Smoktunowicz-Young finitely generated nil algebra of Gelfand-Kirillov dimension $\leq 3$ (see \cite{18}, \cite{19}).

Then the algebra $C'$ in our construction above is finitely generated and nil (though not necessarily stable nil) and $M_\infty(A)\triangleleft_\ell C'$. By Lemma \ref{Lemma7}, we can choose a generating linear transformation $\gamma$ so that \linebreak$w_\gamma(n)\leq n^{d+\epsilon_n}$, $\epsilon_n\rightarrow0$, $n\rightarrow\infty$. Then by Corollary \ref{Corollary1}, \linebreak$GK\dim C'\leq d+6$, which finishes proof of Theorem \ref{Theorem5}.

\begin{remark*}
If we knew that an infinite dimensional stable nil algebra of finite Gelfand-Kirillov dimension exists, then we could embed an arbitrary countable dimensional stable algebraic primitive algebra of finite Gelfand-Kirillov dimension in a 2-generated stable algebraic primitive algebra of finite Gelfand-Kirillov dimension, thus answering the second part of question 7 in \cite{8}.
\end{remark*}

\section{Examples of finitely generated nil algebras of arbitrary Gelfand-Kirillov dimension $d\geq 8$}\label{Section6}

Everywhere in this section, we assume that the ground field $F$ is countable. Let $B$ be an infinite dimensional graded finitely generated Lenagan-Smoktunowicz-Young nil algebra of Gelfand-Kirillov dimension $\leq 3$ (\cite{18}, \cite{19}). Without loss of generality, we will assume that $\ell(B)=\{b\in B|bB=(0)\}=(0)$.

\begin{lemma}\label{Lemma11}
For arbitrary linearly independent elements \linebreak$b_1,\cdots, b_n\in B$ and arbitrary $s\geq1$, there exists an element $b\in B^s$ such that the elements $b_1b,\cdots, b_nb$ are still linearly independent.
\end{lemma}
\begin{proof}
We will induct on $n$. For $n=1$, the assertion of the lemma means that $b_1B^s\neq(0)$, which follows from the assumption on the left annihilator of $B$.

Suppose that the assertion is true for $n-1$. Choose an element $b\in B^s$ such that the elements $b_1b, \cdots, b_{n-1}b$ are linearly independent. Since $\bigcap\limits_{i\geq1}B^i=(0)$, it follows that there exists $t\geq s$ such that $\text{span}_F(b_1b,\cdots, b_{n-1}b)\cap B^t=(0)$.

Again, by the inductive assumption, we can choose an element \linebreak$b'\in B^t$ such that $b_1b',\cdots, b_{n-1}b'$ are linearly independent elements. Assuming that the assertion of the lemma is wrong, there exist scalars $\alpha_1,\cdots,\alpha_{n-1},\beta_1,\cdots,\beta_{n-1}\in F$ such that
\[b_nb=\sum\limits_{i=1}^{n-1}\alpha_ib_ib, \hspace{1cm} b_nb'=\sum\limits_{i=1}^{n-1}\beta_ib_ib'.\]
Since the elements $b_1(b+b'),\cdots, b_{n-1}(b+b')$ are linearly independent, there exist scalars $\gamma_1,\cdots,\gamma_{n-1}\in F$ such that
\[b_n(b+b')=\sum\limits_{i=1}^{n-1}\gamma_ib_i(b+b').\]
Subtracting the first two equations from the third, we get
\[\sum\limits_{i=1}^{n-1}(\gamma_i-\alpha_i)b_ib+\sum\limits_{i=1}^{n-1}(\gamma_i-\beta_i)b_ib'=0.\]
It implies that $\alpha_i=\beta_i=\gamma_i$, $1\leq i\leq n-1$.

We will use the following elementary statement from Linear Algebra:\\
Let $V$ be a vector space over an infinite field. Let $v_1,\cdots,v_n\in V$ be arbitrary elements and let $w_1,\cdots,w_n\in V$ be linearly independent elements. Then there exists a scalar $\xi\in F$ such that the elements $v_i+\xi w_i$, $1\leq i\leq n$, are linearly independent.

This implies that for an arbitrary element $b''\in B^t$, there exists a scalar $\xi\in F$ such that the elements $b_i(b''+\xi b')$, $1\leq i\leq n-1$, are linearly independent. Taking $b''+\xi b'$ instead of $b'$, we get
\[(b_n-\sum\limits_{i=1}^{n-1}\alpha_ib_i)(b''+\xi b')=0,\]
and therefore
\[(b_n-\sum\limits_{i=1}^{n-1}\alpha_ib_i)b''=0, \hspace{1cm} (b_n-\sum\limits_{i=1}^{n-1}\alpha_ib_i)B^t=(0).\]
This contradicts the assumption that the left annihilator of $B$ is zero and completes the proof of the lemma.
\end{proof}

Since the ground field $F$ is countable, it follows that the algebra $B$ is countable.

Let $\bcal$ be the set of all nonempty finite sequences of linearly independent elements of $B$, $\text{card }\bcal=\aleph_0$.

Let $u_1,u_2,\cdots$ be a sequence of elements of $\bcal$ such that each element of $\bcal$ occurs in this sequence infinitely many times.

The algebra $B$ is generated by the homogeneous component of degree 1, $V=B_1$, $V^n=\sum\limits_{i=1}^nB_i$.

We will construct an increasing sequence of integers \linebreak$0=n_0<n_1<n_2<\cdots$. Suppose that $k\geq2$ and $n_0,n_1,\cdots,n_{k-1}$ have already been constructed. Let $u_k=(b_1,\cdots,b_m)\in\bcal$, the elements $b_1,\cdots, b_m$ are linearly independent. By Lemma \ref{Lemma11}, there exists an element $b\in B^{n_{k-1}+1}$ such that the elements $b_1b,\cdots,b_mb$ are linearly independent. Choose $n_k$ such that $n_k>e^{n_{k-1}}$, $n_k>e^{e^k}$, and $b_1b,\cdots, b_mb\in V^{n_{k-1}}$. This completes the construction of the sequence $0=n_0<n_1<n_2<\cdots$.

For an arbitrary $\alpha\geq2$, W. Bohro and H. P. Kraft \cite{9} constructed a graded $F$-algebra $R=\sum\limits_{i=1}^\infty R_i$ generated by two elements $x,y\in R_1$, such that for any $\epsilon>0$ we have
\[n^{\alpha-\epsilon}\leq\dim_F\sum\limits_{i=1}^nR_i\leq n^{\alpha+\epsilon}\]
for all sufficiently large $n$. Let $f(n)=\dim_F\sum\limits_{i=1}^nR_i$.

Let $J$ be a graded ideal of the free associative algebra $F\langle x,y\rangle$ such that $F\langle x,y\rangle/J\cong R$. Now we are ready to introduce a countable dimensional locally nilpotent algebra $A$. Let $X=\{x_1,x_2,\cdots\}$, \linebreak$Y=\{y_1,y_2,\cdots\}$. Consider the algebra $A$ presented by the set of generators $X\cup Y$ and the following set of relations:
\begin{enumerate}[(1)]
\item \label{Enum1} $x_ix_jx_k=0$, where $i,j,k$ are arbitrary, distinct integers;
\item \label{Enum2} $J(x_i,x_j)=(0)$, $i\neq j$, where $J(x_i,x_j)$ is the image of the ideal $J$ under the homomorphism $F\langle x,y\rangle\rightarrow F\langle X,Y\rangle$, $x\rightarrow x_i$, $y\rightarrow x_j$;
\item \label{Enum3} $id_{F\langle X,Y\rangle}(x_i)^{n_{i+3}}=(0)$;
\item \label{Enum4} $[X,y_i]=[Y,y_i]=(0)$, $i\geq 1$;
\item \label{Enum5} $y_i^2=0$, $i\geq 1$.
\end{enumerate}

Let $g_k(n)=\dim_F\sum\limits_{\mu=1}^n(\sum\limits_{i=1}^kFx_i)^\mu$.

\begin{lemma}\label{Lemma12}
Suppose that $n_k\leq n< n_{k+2}$. Then
\[f(n)\leq g_k(n)\leq \binom{k}{2}f(n).\]
\end{lemma}
\begin{proof}
Let $\til{J}$ be the ideal of the free algebra $F\langle X,Y\rangle$ generated by (\ref{Enum1})-(\ref{Enum5}). We have
\[\langle x_{k-1},x_k\rangle\cap\til{J}\subseteq J(x_{k-1},x_k)+\sum\limits_{i=n_{k+2}}^\infty F\langle X,Y\rangle_i.\]
Hence,
\[g_k(n)\geq\dim_F\sum\limits_{i=1}^n(\langle x_{k-1},x_k\rangle/J(x_{k-1},x_k))_i=f(n).\]
On the other hand, relation (\ref{Enum3}) implies
\[F\langle x_1,\cdots, x_k\rangle\subseteq\sum\limits_{1\leq i\neq j\leq k}\langle x_i,x_j\rangle+\til{J}.\]
This implies that $g_k(n)\leq \binom{k}{2}f(n)$ and completes the proof of the lemma.
\end{proof}

Now we will define a generating dense linear transformation \linebreak$\gamma:\hat{B}\rightarrow A$. Recall that $V=B_1$. We let $V^0=F\cdot1$ and \linebreak$\gamma(1)=0$. Suppose that $\gamma:V^{n_{k-1}}\rightarrow A$ has already been defined. Let \linebreak$u_k=(b_1,\cdots,b_m)\in\bcal$. Recall that in the course of choosing the numbers $n_k$, we first chose an element $b\in B^{n_{k-1}+1}$ such that $b_1b,\cdots, b_mb$ are linearly independent and then choose $n_k$ large enough so that $b_1b,\cdots, b_mb\in V^{n_k-1}$. Choose an element $v_k\in V^{n_k}\setminus V^{n_k-1}$ and a subspace $T_k\subset V^{n_k}$ so that $V^{n_k}=V^{n_{k-1}}\oplus T_k\oplus Fb_1b\oplus\cdots\oplus Fb_mb\oplus Fv_k$ is a direct sum of subspaces.

Define $\gamma(T_k)=0$, $\gamma(b_1b)=\cdots=\gamma(b_{m-1}b)=0$, $\gamma(b_mb)=y_k$,\linebreak $\gamma(v_k)=x_k$.

It is clear that $\gamma$ is a generating linear transformation. We will show that $\gamma$ is dense. Choose an element $u=(b_1,\cdots,b_m)\in\bcal$ and a nonzero element $a\in A$. The linearly independent set $u$ occurs infinitely many times in the sequence $u_1,u_2,\cdots$. Choose $k\geq1$ such that\linebreak $u_k=(b_1,\cdots,b_m)$ and $y_k$ does not occur in $a$. Then there exists an element $b\in B$ such that $\gamma(b_1b)=\cdots=\gamma(b_{m-1}b)=0$, $\gamma(b_mb)=y_k$, $ay_k\neq0$.

As above, consider the element $c_\gamma\in Lin(\hat{B},\hat{B}\otimes_FA)$, $c_\gamma(b)=1\otimes\gamma(b)$ for $b\in B$. Let $C=\langle B,c_\gamma\rangle$.

\begin{proof}[Proof of Theorem 6.1]
We proved in \S\ref{Section4} that $C$ is a nil algebra. We will show that $GK\dim C=2GK\dim(B)+\alpha$.

Indeed, by Lemma \ref{Lemma4}, $g_C(n)\sim g_B(n)^2w_\gamma(n)$. We will estimate $w_\gamma(n)$.

Let $n_k\leq n<n_{k+1}$. Then
\[W_n=\text{span}(\gamma(V^{i_1})\cdots\gamma(V^{i_r}),i_1+\cdots+i_r\leq n).\]
Since $\gamma(V^{n_i})\subseteq\gamma(V^{n_{i-1}})+Fx_i+Fy_i$ for $i\geq1$, and all $i_1,\cdots,i_r$ are smaller than $n_{k+1}$, it follows that
\[\gamma(V^{i_1})+\cdots+\gamma(V^{i_r})\subseteq\sum\limits_{i=1}^kFx_i+\sum\limits_{i=1}^{k+1}Fy_i.\]

Then $w_\gamma(u)\leq g_k(n)2^{k+1}$. By Lemma \ref{Lemma12}, for any $\epsilon>0$ and a sufficiently large $n$, we have $g_k(n)\leq n^{\alpha+\epsilon}\binom{k}{2}$. Since $n\geq n_k>e^{e^k}$, it follows that $k<\ln(\ln n)$. Hence, for any $\epsilon'$, $0<\epsilon<\epsilon'$, we have $w_\gamma(n)\preceq n^{\alpha+\epsilon'}$.

On the other hand, $x_1,\cdots, x_{k-1}\in V^{n_{k-1}}$. Hence, $w_\gamma(n)\geq g_{k-1}([\frac{n}{n_{k-1}}])$. We have $n_{k-1}\leq[\frac{n}{n_{k-1}}]<n_{k+1}$. Hence, by Lemma \ref{Lemma12},
\[g_{k-1}([\frac{n}{n_{k-1}}])\geq f([\frac{n}{n_{k-1}}]).\]
From $n\geq n_k\geq e^{n_{k-1}+1}$, we conclude that $n_{k-1}\leq \ln n-1$ and therefore $[\frac{n}{n_{k-1}}]\geq \frac{n}{\ln n}$. Hence, for an arbitrary $\epsilon>0$ for a sufficiently large $n$, we have $w_\gamma(n)\geq(\frac{n}{\ln n})^{\alpha-\epsilon}$. This implies that for any $\epsilon'>\epsilon$, we have $w_\gamma(n)\geq n^{\alpha-\epsilon'}$. This implies that $GK\dim C=2GK\dim(B)+\alpha$ and completes the proof of the theorem.
\end{proof}

\section*{Acknowledgement}

The fourth author gratefully acknowledges the support form the NSF.

\bibliographystyle{amsplain}

\bibliography{MWPOAAETbib}

\endnotetext[1]{Department of Mathematics, King Abdulaziz University, Jeddah, SA,\\
E-mail address: analahmadi@kau.edu.sa; hhaalsalmi@kau.edu.sa;}

\endnotetext[2]{Department of Mathematics, Ohio University, Athens, USA,\\
E-mail address: jain@ohio.edu;}

\endnotetext[3]{Department of Mathematics, University of California, San Diego, USA\\
E-mail address: ezelmano@math.ucsd.edu}

\theendnotes

\end{document}